\documentclass{article}
\usepackage{amssymb,amsfonts,amsmath,amstext,amsthm,amscd,graphics}

\newtheorem{theorem}{Theorem}[section]

\newtheorem{lemma}[theorem]{Lemma}

\theoremstyle{definition}
\newtheorem*{rem}{Remarks}

\newcommand{\R}{\mathbb{R}}
\newcommand{\N}{\mathbb{N}}

 \newcommand{\av}{\arrowvert}

\numberwithin{equation}{section}

\begin{document}
\title{Quasiregular dynamics on the $n$-sphere}
\author{Alastair N. Fletcher
\& Daniel A. Nicks}

\maketitle

\begin{abstract}
In this article, we investigate the boundary of the escaping set $I(f)$ for quasiregular mappings on $\R^{n}$, both in the uniformly
quasiregular case and in the polynomial type case. The aim is to show that $\partial I(f)$ is the Julia set $J(f)$ when the latter is defined, and shares properties with the Julia set when $J(f)$ is not defined.

\end{abstract}

\section{Introduction}

There has been much recent interest in complex dynamics, the study of iteration of analytic functions in the plane. See for example \cite{Beardon}
or \cite{Milnor} for an introduction to the dynamics of rational maps and \cite{Bergweiler1} for an introduction to the transcendental
case. Quasiregular mappings are a natural generalization of analytic functions to higher dimensions, displaying many similar properties. The standard reference for the theory of quasiregular mappings is Rickman's monograph \cite{Rickman}.

While there has been some study of the dynamics of quasiregular mappings (for example \cite{Bergweiler2, BFLM, Siebert}), for the most part this has been restricted to the case of uniformly quasiregular mappings (introduced in \cite{IM}
and see also \cite{HMM, Mayer1, Mayer2}), that is, those quasiregular mappings for which all the iterates have a common bound on the distortion. If all the iterates of a quasiregular mapping $f$ have distortion bounded by $K$, then $f$ is called uniformly $K$-quasiregular (henceforth called $K$-uqr for brevity). This condition allows one to carry over the ideas of Fatou and Julia sets from complex dynamics to quasiregular dynamics, at least in this special case. While these notions may not make sense for an arbitrary quasiregular mapping, the notion of the escaping set always does. The escaping set for a quasiregular mapping $f$ is defined to be
\begin{equation}\label{escapingset}
I(f) := \{ x \in \overline{\R^{n}} : f^{k}(x) \text{ is defined for all } k \in \N, \lim _{k \rightarrow \infty} f^{k}(x) = \infty \}.
\end{equation}
It was proved by Eremenko in \cite{Eremenko} that the escaping set $I(f)$ of a transcendental analytic function in the plane is non-empty and that
the boundary of the escaping set is the Julia set $J(f)$. The same result for meromorphic functions was established by Dominguez in \cite{Dominguez}. In \cite{BFLM} it was shown that if a quasiregular mapping $f:\R^{n} \rightarrow \R^{n}$ grows sufficiently fast then $I(f)$ is non-empty,
and further, $I(f)$ contains an unbounded component.

These results raise the question of whether the boundary of the escaping set of an arbitrary quasiregular mapping possesses properties typically associated with a Julia set, even though the Julia set may not be defined. In this paper we will show that the boundary $\partial I(f)$ is perfect for some classes of quasiregular mappings from $\R^{n}$ to itself; that is, it contains no isolated points. It is well known that the Julia set of an analytic or meromorphic function is perfect.

A quasiregular mapping $f:\R^{n}\to\R^{n}$ is said to be of polynomial type if $ f(x)  \rightarrow \infty$ as $ x  \rightarrow \infty$, whereas $f$ has an essential singularity at infinity if this limit does not exist. Note that by (\ref{escapingset}), the point at infinity is contained in $I(f)$ if $f$ is of polynomial type, but not if $f$ has an essential singularity at infinity.
The degree of a polynomial type mapping may be thought of as a generalization of the degree of a polynomial. It can be defined by
\begin{equation}\label{eqn:def degree}
\operatorname{deg}(f):=\sup_{y\in\R^n}|f^{-1}(y)|,
\end{equation}
that is, the maximal number of pre-images of any value in $\R^n$. It is well known that $f$ is of polynomial type if and only if (\ref{eqn:def degree}) is finite (see \cite{Rickman}, as well as \cite{HK}, for further properties of polynomial type mappings).
The definition of the inner dilatation $K_I$ of a quasiregular map is postponed until the next section.
We then have the following result for quasiregular mappings of polynomial type.

\begin{theorem}
\label{polyqr}
Let $n \geq 2$ and $f:\R^{n} \rightarrow \R^{n}$ be $K$-quasiregular of polynomial type. If the degree of $f$ is greater than $K_{I}$, then $I(f)$ is a non-empty open set and $\partial I(f)$ is perfect.
\end{theorem}

We collect the following results on $I(f)$ and $\partial I(f)$.

\begin{theorem}
\label{further}
Let $f:\R^{n}\rightarrow \R^{n}$ be $K$-quasiregular of polynomial type and suppose that the degree of $f$ is greater than $K_{I}$.
\begin{enumerate}
\item For any $k \geq 2$ we have $I(f^{k}) = I(f)$.
\item  The family of iterates $\{ f^{k}: k \in \N \}$ is equicontinuous
on $I(f)$ and not equicontinuous at any point of $\partial I(f)$, with respect to the spherical metric on $\overline{\R^n}$.
\item $\partial I(f)$ is infinite.
\item $I(f), \partial I(f)$ and $\R^{n} \setminus \overline{I(f)}$ are completely invariant.
\item $I(f)$ is connected.
\end{enumerate}
\end{theorem}

Finally we show that the boundary of the escaping set coincides with the Julia set for uniformly quasiregular mappings. The fact that the Julia set of a uniformly quasiregular mapping in $\overline{\R ^{n}}$ is perfect is known and proved in \cite{Sthesis}. We provide a proof for the convenience of the reader.

\begin{theorem}
\label{uqr}
Let $n \geq 2$ and $f:\R^{n} \rightarrow \R^{n}$ be a $K$-uqr mapping which is not injective. Then $\partial I(f) = J(f)$ and is an infinite perfect set.
\end{theorem}

In particular, a non-injective uniformly quasiregular map has a non-empty escaping set (see Lemma \ref{lemma5.1}).

In view of Theorem \ref{polyqr} and results of \cite{BFLM} stating that $I(f)$ is non-empty if $f$ is a quasiregular mapping with an essential singularity at infinity, it is natural to ask the following question in analogy to the case of transcendental entire functions. \emph{If $f:\R^{n} \rightarrow \R^{n}$ is a $K$-quasiregular mapping with an essential singularity at infinity, then is the boundary of $I(f)$ always a perfect set?}

See remarks (iii) and (iv) below for partial results in this direction.

\begin{rem}\mbox{}
\begin{enumerate}
\item Theorem \ref{polyqr} is sharp as the following example shows (see \cite[p.13]{Rickman}).
Let $n \geq 2$, $k \in \N$ and consider the mapping $f:(r, \varphi, y) \mapsto (r, k \varphi, y)$ in cylindrical coordinates in $\R^{n}$ (i.e. $y \in \R^{n-2}$). The branch set of $f$ is the $(n-2)$-dimensional hyperplane defined by $r=0$.
The degree of $f$ can be shown to be $k$, and further $K_{I}(f) = k$. However $\av f(x) \av = \av x \av$ for all $x \in \R^{n}$, and so $I(f)\cap\R^n$ is empty.
\item
By \cite{Hinkkanen}, every uniformly quasiregular mapping $f:\R^{2} \rightarrow \R^{2}$ can be conjugated via a quasiconformal map $h:\R^2 \rightarrow \R^2$ to an analytic function $g$, that is $f = h^{-1} \circ g \circ h$. In this case, the conclusions of Theorem \ref{uqr} follow from the standard analogous results for analytic functions.
\item
When $n \geq 3$, all the known examples of uniformly quasiregular maps are of polynomial type. It is an interesting question as to whether there exist any uniformly quasiregular maps with an essential singularity at infinity, since then Theorem \ref{uqr} would give quasiregular maps which are not of polynomial type for which $\partial I(f)$ is perfect.
\item
In \cite{Bergweiler3} it is shown that for certain Zorich-type maps ($n$-dimensional versions of the exponential map), the set of points  which do not converge to a certain fixed point form Devaney hairs. It follows from the proof of this that the escaping set consists of these hairs together with some of the endpoints, and as such, the boundary of the escaping set of these maps must be perfect.
\item
If $f$ is allowed to have poles, while still having finite degree, then $f$ is said to be quasirational. In this case, provided that $f$ fixes infinity and the topological index of $f$ at infinity is greater than $K_{I}$, the methods of Theorem \ref{polyqr} remain valid and give an unbounded component of the escaping set
and also show that $\partial I(f)$ is perfect. In this case $I(f)$ may no longer be connected, since it could contain bounded components consisting of neighbourhoods of poles of $f$.
\item
The proof of Theorem \ref{polyqr} shows that $\partial I(f)$ is bounded when $f$ is of polynomial type and the degree of $f$ is greater than $K_{I}$. If $f$ has an essential singularity at infinity, then $\partial I(f)$ is unbounded. To prove this, observe that $I(f)$ is unbounded by its non-emptiness \cite{BFLM}. To see that the complement of $I(f)$ is unbounded, note that $f$ has infinitely many $2$-periodic points by \cite{Bergweiler2}. Then the Big Picard Theorem for quasiregular mappings (see Theorem~2.27 of \cite[p.87]{Rickman}) shows that for any $R>0$, the domain $\{x \in \R^{n}:\av x \av >R\}$ contains a pre-image of one of these periodic points. This is in direct analogy with the Julia set of an analytic function on the plane being bounded or unbounded respectively in the polynomial and transcendental cases.
\item
In \cite{DL}, a Julia set for quasiregular mappings of polynomial type in dimension $2$ is investigated, although there a quasiregular map is defined as a composition of a quasiconformal and an analytic map. This decomposition is not available in dimensions greater than $2$. The Julia set is defined to be the set of points $z$ for which any neighbourhood $U$ of $z$ has the property that
\begin{equation*}
\overline{\R^{2}} \setminus \{a,b \} \subset \bigcup _{k \geq 0} f^{k}(U) ,
\end{equation*}
where $a,b$ are two possible exceptional values, independent of $z$.
\end{enumerate}
\end{rem}

\section{Results needed for the proofs}

Write $\av x \av = \av (x_{1},...,x_{n}) \av = \left( \sum _{i=1}^{n} x_{i}^{2} \right) ^{1/2}$ for the Euclidean norm in $\R^{n}$ and $B(x,r) = \{ y \in \R^{n}: \av x - y \av < r\}$ for the ball of radius $r$ in this norm.
A continuous mapping $f:G \rightarrow \R^{n}$ from a domain $G\subseteq \R^{n}$ is called quasiregular if $f$ belongs to the Sobolev space $W_{n,loc}^{1}(G)$ and there exists $K \in [1, \infty)$ such that
\begin{equation}
\label{eq2.1}
\av f'(x) \av ^{n} \leq K J_{f}(x)
\end{equation}
almost everywhere in $G$. Here $J_{f}(x)$ denotes the Jacobian determinant of $f$ at $x \in G$. The smallest constant $K \geq 1$ for which (\ref{eq2.1}) holds is called the outer dilatation $K_{O}(f)$ of $f$. If $f$ is quasiregular, then we also have
\begin{equation}
\label{eq2.2}
J_{f}(x) \leq K' \inf_{\av h \av = 1} \av f'(x)h \av^{n}
\end{equation}
almost everywhere in $G$ for some $K' \in[1, \infty)$.
The smallest constant $K' \geq 1$ for which (\ref{eq2.2}) holds is called the inner dilatation $K_{I}(f)$ of $f$. The maximal dilatation $K=K(f)$ of $f$ is the larger of $K_{O}(f)$ and $K_{I}(f)$. A map is called $K$-quasiregular if $K(f) \leq K$. It is well known that quasiregular mappings are open and discrete \cite{Reshetnyak}. For further details on the basic theory of quasiregular mappings, we refer to \cite{Rickman}.

If $f:G \rightarrow \overline{\R^{n}} = \R^{n} \cup \{ \infty \}$ is continuous, then $f$ is called quasimeromorphic if each $x \in G$ has a neighbourhood $U_{x}$ such that either $f$ or $g \circ f$ is a quasiregular map of $U_{x}$ into $\R^{n}$, where $g$ is a sense-preserving M\"{o}bius map of $\overline{\R^{n}}$ with $g(\infty) \in \R^{n}$.
For such a quasimeromorphic map the topological index of $f$ at $x\in G$ is denoted by $i(x,f)$ and may be defined as the infimum of
\[ \sup_{y\in\overline{\R^n}}|f^{-1}(y)\cap U| \]
when $U$ runs through all neighbourhoods of $x$.

It is clear that a polynomial type mapping can be extended to a continuous open mapping from $\overline{\R^{n}}$ to itself which fixes the point at infinity. The next result summarizes some basic properties of polynomial type maps.

\begin{lemma}\label{deglem}
Let $f$ be a quasiregular mapping of polynomial type of degree $d$. Then $f$ does not omit any value in $\R^n$ and the iterate $f^k$ has degree $d^k$. Moreover, if a point $x$ is such that $f^{-1}(x)=\{x\}$, then $i(x,f)=d$.
\end{lemma}

\begin{proof}
The fact that $f$ takes every value in $\R^n$ follows from the observation that the image of $\overline{\R^n}$ under a continuous open map is both open and compact.

The degree and topological index are both defined in terms of induced mappings of homology groups in \cite{Rickman}. The equivalence of these definitions to the ones given above follows from Proposition 4.10 (2), (4) of \cite[p.19]{Rickman}. Using the homology definitions the conclusions of the lemma follow from \cite[\S I.4.2]{Rickman}.
\end{proof}

Another well-known fact is that quasiregular mappings are H\"{o}lder continuous. For our purposes, we will use a local version of H\"{o}lder continuity due to Martio \cite{Martio}.

\begin{theorem}[\cite{Martio}, see also {\cite[p.72]{Rickman}}]
\label{localholderthm}
Let $f:G\rightarrow \R^{n}$ be quasiregular and non-constant, and let $x \in G$. Then there exist positive constants $r$ and $C$ such that for $\av x - y \av <r$,
\begin{equation}
\label{localholder}
\av f(y) - f(x) \av \leq C \av y - x \av ^{\alpha}
\end{equation}
where
\begin{equation}
\label{localholder2}
\alpha = \left ( \frac{i(x,f)}{K_{I}(f)} \right ) ^{1/(n-1)}.
\end{equation}
\end{theorem}

One of the main reasons for viewing quasiregular mappings as higher dimensional analogues of analytic functions in the plane is the following result of Rickman.

\begin{theorem}[\cite{Rickman}]
\label{rickman}
For every $n \geq 3$ and $K \geq 1$, there exists a positive integer $q=q(n,K)$ which depends only on $n$ and $K$, such that the following holds. Every $K$-quasimeromorphic mapping $f:\R^{n} \rightarrow \overline{\R^{n}} \setminus \{ a_{1},...,a_{m} \}$ is constant whenever
$m \geq q$ and $a_{1},...,a_{m}$ are distinct points in $\overline{\R^{n}}$.
\end{theorem}

This leads to a version of Montel's Theorem for quasiregular maps.

\begin{theorem}[\cite{Miniowitz}]
\label{montel}
Let $\mathcal{F}$ be a family of $K$-quasimeromorphic mappings in a domain $G \subset \R^{n}$ and let $q=q(n,K)$ be Rickman's constant from Theorem \ref{rickman}. If there exist distinct points $a_{1},...,a_{q} \in \overline{\R^{n}}$ such that $f(G) \cap \{ a_{1},...,a_{q} \} = \emptyset$ for all $f \in \mathcal{F}$, then $\mathcal{F}$ is a normal family.
\end{theorem}

\section{Proof of Theorem \ref{polyqr}}

Let $f:\R^{n} \rightarrow \R^{n}$ satisfy the hypothesis of the theorem with degree $d$. Since $f$ is of polynomial type, it can be extended to a mapping from $\overline{\R^{n}}$ to itself which fixes infinity. It is natural to proceed by conjugating this fixed point to the origin and then applying Theorem \ref{localholderthm}. To this end, let $A:\overline{\R^{n}} \rightarrow \overline{\R^{n}}$ be the inversion mapping in the unit sphere given by
\begin{equation*}
A(x) = \frac{x}{\av x \av ^{2}}.
\end{equation*}
Note that $A$ is sense-reversing, $A^{-1} = A$ and $A(0)=\infty$. Here and throughout, we use $0 \in \R^{n}$ to denote the point $(0,...,0)$. Consider the mapping
\begin{equation*}
g = A \circ f \circ A
\end{equation*}
from $\overline{\R^{n}}$ to itself which fixes $0$. By construction, $g$ is sense-preserving and an elementary calculation shows that $g$ is quasiregular with degree $d$ and $K_{I}(g) = K_{I}(f)$.

Applying Theorem \ref{localholderthm} to $g$ and $x=0$, we have from (\ref{localholder}) that
there exist constants $r>0$ and $C>0$ such that
\begin{equation}
\label{eq3.1}
\av g(y) \av \leq C \av y \av ^{\alpha}
\end{equation}
for $\av y \av <r$ and
\begin{equation}
\label{eq3.2}
\alpha = \left ( \frac{i(0,g)}{K_{I}(g)} \right ) ^{1/(n-1)}.
\end{equation}

Using Lemma \ref{deglem}, the fact that $g^{-1}(0)=\{0\}$, and the hypothesis of the theorem, we find that $i(0,g)=d>K_I(g)$ and so (\ref{eq3.2}) gives $\alpha>1$. Since $\av A(y) \av = \av y \av ^{-1}$, by substituting $f$ back into (\ref{eq3.1}) and writing $x=A(y)$ we obtain
\begin{equation}
\label{eq3.4}
\av f(x) \av \geq C^{-1} \av x \av ^{\alpha}
\end{equation}
for $\av x \av >R=1/r$.
Take
\begin{equation}
\label{eq3.4b}
R' = \max \{ R, (2C)^{1/(\alpha - 1)} \}.
\end{equation}
Then (\ref{eq3.4}) gives that
\begin{equation}
\label{eq3.4a}
\av f(x) \av > 2 \av x \av
\end{equation}
for $\av x \av > R'$, and so $f^k(x)\to\infty$ as $k\to\infty$. This implies that infinity is an attracting fixed point of $f$ and the basin of attraction includes $\{ x \in \R^{n} : \av x \av > R' \}$.
Consequently we have that $I(f)$ is non-empty, and further, that $I(f)$ contains a neighbourhood of infinity.

We show next that $I(f)$ is open. Let $x\in I(f)\cap\R^n$ and choose $n_0$ so that $|f^{n_0}(x)|>R'$. Then there exists $\epsilon>0$ such that
\begin{equation*}
\Omega : = B(f^{n_{0}}(x), \epsilon ) \subseteq \{y:|y|>R'\} \subseteq I(f).
\end{equation*}
Since quasiregular maps are continuous we can choose $\delta$ so small that
\[ f^{n_0}(B(x,\delta))\subseteq \Omega. \]
Then $B(x,\delta)\subseteq I(f)$ by the complete invariance of the escaping set.

It follows from the fact that $I(f)$ is open that $I(f)$ has no isolated points. To show that $\partial I(f)$ is perfect it just remains to show that the complement of $I(f)$ also has no isolated points. For the sake of contradiction we suppose that there exist $x\notin I(f)$ and $\delta>0$ such that $B(x,2\delta)\setminus\{x\}\subseteq I(f)$.
 Note that we must have $\av f^k(x) \av < R'$ for all $k$. Let $E = \partial B(x, \delta)$ and for $j \in \N$ define
\begin{equation}
\label{eq3.5a}
E_{j} = \{ y \in E : \av f^{j}(y) \av > R' \}.
\end{equation}
Since $f$ is continuous, each set $E_j$ is open and by (\ref{eq3.4a}) we have that $E_{j} \subseteq E_{j+1}$. Further, because $E \subset I(f)$, we see that
\begin{equation*}
E= \bigcup _{j=1}^{\infty} E_{j}.
\end{equation*}
Hence the $E_j$ form a nested open cover of the compact set $E$ and so we can find $N\in\N$ such that $E_N=E$. That is, $f^N(E) \subset \{ y: \av y \av > R' \}$. We now claim that $B(0,R')\subseteq f^N(B(x,\delta))$. Otherwise, using that $f^N(x)\in B(0,R')\cap f^N(B(x,\delta))$ we obtain a point of $\partial f^N(B(x,\delta))$ that lies in $B(0,R')$, contradicting the fact that $\partial f^N(B(x,\delta))\subseteq f^N(E) \subseteq \{y:|y|>R'\}$ because $f^N$ is an open map.

Therefore every point of $B(0,R')$ is the image of some point in $B(x,\delta)$ under $f^N$. In particular, since $x\in B(0,R')$ and $I(f)$ is completely invariant we must have that $f^N(x)=x$. Then every point of $B(0,R')\setminus\{x\}$ is the image of some point in $I(f)$ under $f^N$, and it follows that every point of $\R^n\setminus\{x\}$ escapes.

As $x$ is the only non-escaping point it must be a fixed point of $f$, and so by Theorem \ref{localholderthm} there exist constants $r>0$ and $C>0$ such that
\begin{equation}
\label{eq3.5}
\av  f(y) - x \av \leq C \av y - x \av ^{\alpha}
\end{equation}
for $\av y -x \av <r$. We have $\alpha >1$ here by (\ref{localholder2}) and the fact that $i(x,f) = d>K_{I}(f)$ by Lemma~\ref{deglem}. Since $\alpha >1$, if we choose $\eta >0$ small enough, (\ref{eq3.5}) implies that $\av f(y) - x \av < \av y -x \av$ for $\av y - x \av < \eta$.
This contradicts the fact that every point of $\R^{n} \setminus \{ x \}$ escapes and completes the proof that $\partial I(f)$ is perfect.

\section{Proof of Theorem \ref{further}}

\begin{description}
\item{\bf Part (i).} It is clear that $I(f)\subseteq I(f^k)$. For $x\in I(f^k)$ and $j \in \{1,...,k-1\}$ consider $f^{mk + j}(x)$. We have, by the continuity of $f$ on $\overline{\R^{n}}$,
\begin{equation*}
\lim _{m \rightarrow \infty} f^{mk+j}(x) = \lim_{m \rightarrow \infty} f^{j}(f^{mk}(x)) = f^{j}\left( \lim _{m\rightarrow\infty} f^{mk}(x) \right) = f^{j}(\infty) = \infty
\end{equation*}
for each $j$. This implies that $x \in I(f)$, and therefore $I(f^{k}) = I(f)$.

\item{\bf Part (ii).} Let $x \in I(f)$ and $V$ be a compact neighbourhood of $x$ such that $V \subset I(f)$, which we can find since $I(f)$ is open. Recall the definition of $R'$ from (\ref{eq3.4b}) and define $V_{j}$ analogously to (\ref{eq3.5a}) by
\begin{equation*}
V_{j} = \{ y \in V : \av f^j(y) \av > R'\}.
\end{equation*}
By the same method as used after (\ref{eq3.5a}), we can find $N$ such that $V_N = V$. This fact, together with (\ref{eq3.4a}), implies that $f^{N+k}(V) \subset \{ \av w \av > 2^{k} R' \}$ so that $f^{k}$ tends to infinity uniformly on $V$.

Now let $x \in \partial I(f)$ and $U$ be an open neighbourhood of $x$. We can find $y \in U \cap I(f)$ and $z \in U \cap (\R^{n} \setminus I(f))$. We simply observe that $f^k(y)\to\infty$ as $k\to\infty$, while $|f^k(z)|\le R'$ for all $k$. This implies that $\{f^{k} \}$ cannot be equicontinuous on $U$, and since $U$ was arbitrary, $\{f^{k} \}$ is not equicontinuous at $x$.

\item{\bf Part (iii).} By Theorem \ref{polyqr}, we have that $I(f)$ is non-empty and $\partial I(f)$ is perfect. Therefore, to show that $\partial I(f)$ is an infinite set, we have to show that there is a non-escaping point. The quickest way to see this is to observe that if $I(f)=\overline{\R^{n}}$ then, because we have locally uniform convergence on $I(f)$ by part (ii), we can find $N$ such that
\[ f^N\!\left(\overline{\R^n}\right)\subseteq\{x\in\R^n:|x|>1\}.\]
 This contradicts the fact that the polynomial type map $f^N$ takes every value in $\R^n$.

In fact we can show rather more, namely that $f$ must have a fixed point in $\R^n$. Using again the definition of $R'$ from (\ref{eq3.4b}), choose $S>R'$ large enough so that $f^{-1} \left ( B(0,S) \right )$ has only one connected component $U$.
By (\ref{eq3.4a}), we have that $\overline{U} \subset B(0,S)$. Since $f$ has finite degree and $U$ contains all the pre-images of points of $B(0,S)$,
we have that $f\av _{U}$ is a proper map
(see \cite[Lemma 2.1.4]{Siebert}) and so by applying \cite[Lemma 2.1.5]{Siebert}, we see that $f$ has a fixed point in $U$.

\item{\bf Part (iv).} It is clear that $I(f)$ is completely invariant.
If $x \in \partial I(f)$, then it is also easy to see that any neighbourhood of $f(x)$ contains points of $I(f)$ and also of $\R^{n} \setminus I(f)$ and the same is true of any $y \in f^{-1}(x)$. Therefore $\partial I(f)$ is completely invariant. Since $I(f)$ and $\partial I(f)$ are completely invariant, it follows that $\R^{n} \setminus \overline{I(f)}$ is completely invariant since $f$ is surjective.

\item{\bf Part (v).} Suppose that $U_{0}$ is the unbounded component of $I(f)$ and $U_{1}$ is a bounded component of $I(f)$. Then since $f$ is surjective, $U_{1} \subset I(f)$ and $I(f)$ is completely invariant, it follows that there exists $j \in \N$ such that $f^{j}(U_{1}) = U_{0}$. Therefore there must be a pole of $f^{j}$ in $\overline{U_{1}}$, which is a contradiction since $f$ has no poles.
\end{description}

\section{Proof of Theorem \ref{uqr}}

We first show that the hypotheses of the theorem imply that $I(f)$ is non-empty.

\begin{lemma}
\label{lemma5.1}
Let $f:\R^{n} \rightarrow \R^{n}$ be a $K$-uqr mapping which is not injective. Then $I(f)\cap\R^n$ is non-empty.
\end{lemma}

\begin{proof}
If $f$ has an essential singularity at infinity, then $I(f)\cap\R^n \neq \emptyset$ by results of \cite{BFLM}. We may therefore assume that $f$ is of polynomial type of degree $d\geq 2$ since $f$ is not injective. Then the degree of $f^{k}$ is $d^{k}$ by Lemma \ref{deglem}. Since $f^{k}$ is $K$-quasiregular, we can choose $k$ large enough so that $d^{k}>K\geq K_{I}$. By Theorem \ref{polyqr}, we have that $I(f^{k})\cap\R^n$ is non-empty.
Part (i) of Theorem \ref{further} then implies that $I(f)\cap\R^n$ is non-empty.
\end{proof}

In this section we use standard terminology from complex dynamics: $J(f)$ is the Julia set, $F(f)$ is the Fatou set,
\begin{equation*}
O^{+}(x) = \bigcup _{k \geq 1} f^{k}(x)
\end{equation*}
is the forward orbit of a point $x$ and
\begin{equation*}
O^{-}(x) = \bigcup _{k \geq 1} f^{-k}(x)
\end{equation*}
is the backward orbit of a point $x$.
Note that in this theorem, we are not assuming that $f$ is of polynomial type and hence defined at infinity. If $f$ is defined at infinity, then by Theorem \ref{further} (ii), infinity is contained in the Fatou set $F(f)$ and so $J(f)$ and $\partial I(f)$ are both bounded. If $f$ has an essential singularity at infinity, then according to the definition of $I(f)$ in (\ref{escapingset}), infinity is not in $I(f)$ and it is in neither $J(f)$ nor $F(f)$. Hence there is no ambiguity in considering both cases at the same time, where $f$ is defined or not at infinity.

\begin{lemma}
Let $f:\R^{n} \rightarrow \R^{n}$ be a $K$-uqr mapping which is not injective. Then we have
\begin{equation*}
J(f) = \partial I(f).
\end{equation*}
\end{lemma}

\begin{proof}
The proof is based on the case of entire functions in the plane, see \cite{Eremenko}.

Let $x_{0} \in \partial I(f)$ and assume that $x_{0} \in F(f)$. Since $F(f)$ is open, there exists a neighbourhood $U_{0}$ of $x_{0}$ such that $U_{0} \subset F(f)$. Since $\{f^{k} \av _{U_{0}} : k \in \N \}$ is a normal family of $K$-quasiregular mappings and $U_0$ contains points of $I(f)$, it
follows that $f^{k} \rightarrow \infty$ in $U_{0}$. Therefore $U_{0} \subset I(f)$, contradicting the fact that $x_{0} \in \partial I(f)$ and so $\partial I(f) \subset J(f) $.

Now suppose that $x_{1} \in J(f)$ and $U_{1}$ is an open neighbourhood of $x_{1}$. By Lemma \ref{lemma5.1} there exists a point $y_{1} \in I(f)\cap\R^n$. For
$j=2,...,q$, define
\begin{equation*}
y_{j} = f(y_{j-1}),
\end{equation*}
where $q=q(n,K)$ is Rickman's constant from Theorem \ref{rickman}. Note that the $y_{j}$ are distinct because $y_{1} \in I(f)\cap\R^n$. Since $x_{1} \in J(f)$ the family $\{ f^{k} \av _{U_{1}}: k \in \N \}$ is not normal. By Theorem~\ref{montel}, $f^{k}(U_{1})$ meets the set $\{ y_{1},...,y_{q}\}$ for infinitely many values $k$. Therefore, there exist $w \in U_{1}$ and $k$ such that $f^{k}(w) = y_{j}$ for some $j \in \{1,...,q\}$. Thus $w \in U_{1}\cap I(f)$. This means that there exist points of $I(f)$ arbitrarily close to points of $J(f)$ and so $J(f) \subset \overline{I(f)}$.

Finally, let $x_{2} \in \text{int}(I(f))$ and $U_{2}$ be a neighbourhood of $x_{2}$ such that $U_{2} \subset I(f)$. Let $\mathcal{F} = \{ f^{k}|_{U_2} : k\in\N \}$. Since $U_{2}$ is contained in $I(f)$, it follows that none of the members of $\mathcal{F}$ have a fixed point in $U_{2}$. By Theorem 3.3.6 of \cite{Siebert}, it follows that $\mathcal{F}$ is normal and hence that $x_{2} \in F(f)$. This shows that $\text{int}(I(f)) \subset F(f)$ and we can conclude that $\partial I(f) = J(f)$.
\end{proof}

A consequence of Montel's Theorem (Theorem \ref{montel}) is that, following \cite{HMM}, we can define the exceptional set $E(f)$  of a $K$-uqr mapping to be the largest discrete completely invariant set such that $E(f)$ has the following properties: for any open set $U$ with $U \cap J(f) \neq \emptyset$ we have
\begin{equation*}
\R^{n} \setminus E(f) \subset \bigcup _{k \geq 0} f^{k}(U)
\end{equation*}
and for every point $x \notin E(f)$, we have
\begin{equation*}
J(f) \subset \overline{ O^{-}(x) }.
\end{equation*}
See \cite{HMM} for more details on the exceptional set, and in particular the following facts: $E(f)$ cannot contain more than $q=q(n,K)$ points, where again $q$ denotes Rickman's constant; $E(f)$ is contained in $F(f)$; and $E(f)$ contains those points whose backward orbit is finite. To finish the proof of Theorem \ref{uqr}, we need the following lemma, the proof of which is standard (cf. \cite{Bergweiler1,Sthesis}).

\begin{lemma}
Let $f:\R^n\to\R^n$ be a $K$-uqr mapping which is not injective. The Julia set $J(f)$ is equal to $\overline{O^{-}(x)}$ for any $x \in J(f)$, is an infinite set and is perfect.
\end{lemma}

\begin{proof}
Let $x \in J(f)$. Then $f^{-k}(x)$ is contained in $J(f)$ for each $k$, and since $J(f)$ is closed and completely invariant, it follows that
\begin{equation}
\label{eq4.1}
\overline{ O^{-}(x) } \subset J(f).
\end{equation}
As $x\notin E(f)$ we have equality in (\ref{eq4.1}).

The Julia set $J(f)$ is non-empty as $f$ is not injective \cite{IM}. Now, if $J(f)$ were finite, then it would consist of points whose backward orbits are finite. However, all such points are in $E(f)$ and hence in $F(f)$, which is a contradiction and so $J(f)$ must be an infinite set.

Choose $x_{0} \in J(f)$ and define $x_1$ as follows. If $x_{0}$ is not periodic, then set $x_{1}= x_{0}$. If $x_{0}$ is periodic, then the forward orbit $O^{+}(x_{0})$ is finite and since the backward orbit $O^{-}(x_{0})$ is infinite, we can choose $x_{1} \in O^{-}(x_{0}) \setminus O^{+}(x_{0})$ which is not periodic. In either case, it follows that $x_{1} \notin O^{-}(x_{1})$. Since $x_{1} \in \overline{O^{-}(x_{1})}$,
it follows that $x_{1}$ is not an isolated point of $J(f)$. Since $J(f)=\overline{O^{-}(x_{1})}$, it follows that no point of $J(f)$ is isolated and $J(f) = \partial I(f)$ is a perfect set.
\end{proof}

Department of Mathematics, \\ University Gardens, \\ University of Glasgow, \\G12 8QW. \\
email:\,{\tt a.fletcher@maths.gla.ac.uk}
\\ \\
School of Mathematical Sciences, \\University of Nottingham, \\
University Park, Nottingham, \\NG7 2RD.\\
email:\,{\tt Dan.Nicks@maths.nottingham.ac.uk}

\end{document}